\documentclass[11pt]{amsart}
\usepackage{amsmath}
\usepackage{amssymb}
\usepackage{amsthm}
\usepackage{enumerate}
\usepackage[mathscr]{eucal}
\usepackage{graphicx}
\usepackage{subfig}
\usepackage{color}

\newcommand{\be}{\begin{equation}}
\newcommand{\ee}{\end{equation}}
\newcommand{\bal}{\begin{aligned}}
\newcommand{\eal}{\end{aligned}}
\newcommand{\bee}{\begin{equation*}}
\newcommand{\eee}{\end{equation*}}

\newtheorem{thm}{Theorem}

\newtheorem{cor}[thm]{Corollary}

\numberwithin{equation}{section}

\begin{document}

\title[Harmonic manifolds]{Characterizing the harmonic manifolds by the eigenfunctions of the Laplacian}
\author[J. Choe]{Jaigyoung Choe}
\author[S. Kim]{Sinhwi Kim}
\author [J. H. Park]{JeongHyeong Park}
\address{Korea Institute for Advanced Study, 85 Hoegiro, Dongdaemungu, Seoul 02455\\
  Korea}
 \email{choe@kias.re.kr}
  \address{
 Department of Mathematics,
 Sungkyunkwan University, 2066 Seobu-ro, Jangan-gu,
 Suwon 16419,
  Korea}
\email{parkj@skku.edu, kimsinhwi@skku.edu}

\thanks{ J.C. supported in part by NRF 2011-0030044, SRC-GAIA, S.K. and J.P. supported by Basic Science Research
Program through the National Research Foundation of Korea(NRF)
funded by the Ministry of Education(NRF-2016R1D1A1B03930449).}


\begin{abstract}
The space forms, the complex hyperbolic spaces and the quaternionic hyperbolic spaces are characterized as the harmonic manifolds with specific radial eigenfunctions of the Laplacian.
\end{abstract}

 \maketitle
\section{Introduction}
In the Euclidean space $\mathbb R^n$ it is well-known hat the Laplace operator $\Delta$ is invariant under orthogonal transformations. Hence $\mathbb R^n$ has the property that the Laplacian of a radial function (function depending only on the distance to the origin) is still radial. Then, is a Riemannian manifold $M$ with this property necessarily $\mathbb R^n$ or the space form? In regard to this interesting question, a harmonic manifold is introduced.

A complete Riemannian manifold $M$ is called {\it harmonic} if it satisfies one of the following equivalent conditions:
\begin{enumerate}
\item For any point $p\in M$ and the distance function $r(\cdot):={\rm dist}(p,\cdot)$, $\Delta r^2$ is radial;
\item For any $p\in M$ there exists a nonconstant radial harmonic function in a punctured neighborhood of $p$;
\item Every small geodesic sphere in $M$ has constant mean curvature;
\item Every harmonic function satisfies the mean value property \cite{Wi};
\item For any $p\in M$ the volume density function $\omega_p=\sqrt{{\rm det}g_{ij}}$ in normal coordinates centered at $p$ is radial.
\end{enumerate}

Lichnerowicz conjectured that every harmonic manifold $M^n$ is flat or rank 1 symmetric. This conjectue has been proved to be true for dimension $n\leq5$ \cite{B,L,N,Sz,Wa}. But Damek and Ricci \cite{DR} found that there are many {{counterexamples}} if dimension $n\geq7$. Euh, Park and Sekigawa \cite{EPS} provide {{a new}} proof of the Lichnerowicz
conjecture for {{dimension $n = 4, 5$ in a slightly}} more general setting using universal curvature identities.

In order to further characterize harmonic manifolds, Shah \cite{Sh}, Szab\'o \cite{Sz} and Ramachandran-Ranjan \cite{RR} paid attention to the volume density function $\omega_p(r)$ as defined in the equivalent condition (5) above. Shah proved that a harmonic manifold with the same volume density as $\mathbb R^n$ is flat, Szab\'o showed $\mathbb S^n$ is the only harmonic manifold with $\omega_p(r)=\frac{1}{r^{n-1}}\sin^{n-1}r$ and Ramachandran-Ranjan showed that a noncompact simply connected harmonic manifold $M^n$ with $\omega_p(r)=\frac{1}{r^{n-1}}\sinh^{n-1}r$ is $\mathbb H^n$. Ramachandran-Ranjan also proved that a noncompact simply connected K\"{a}hler harmonic manifold $M^{2n}$ with $\omega_p(r)=\frac{1}{r^{2n-1}}\sinh^{2n-1}r\cosh r$ is isometric to the complex hyperbolic space. A similar theorem was proved for the quaternionic hyperbolic space as well.

In this paper we remark the fact that the Laplacian of specific radial functions are very simple in space forms.  It is well known that in $\mathbb R^n$
\begin{equation}\label{1}
\Delta r^{2-n}=0\,\,\,\,{\rm and}\,\,\,\,\Delta r^2=2n;
\end{equation}
in $\mathbb S^n$ and $\mathbb H^n$ \cite{CG},
\begin{equation}\label{2}\Delta\cos r=-n\cos r\,\,\,\,{\rm and}\,\,\,\,\Delta \cosh r=n\cosh r,\,\,\,\,{\rm respectively};\end{equation}
and for some hypergeometric function $f$ on $\mathbb CH^n$ and $\mathbb QH^n$,
\begin{equation}\label{3}\Delta f=4(n+1)f\,\,\,\,{\rm and} \,\,\,\,\Delta f=8(n+1)f,\,\,\,\,{\rm respectively}.\end{equation}
Motivated by this fact, we characterize harmonic manifolds in terms of these radial functions. It will be proved that if a radial harmonic function defined in a punctured neighborhood of a harmonic manifold $M$, as in the equivalent condition (2) above, is the same as the radial Green's function of a space form, $\mathbb CH^n$ or $\mathbb QH^n$, then $M$ is the space form, $\mathbb CH^n$ or $\mathbb QH^n$, repectively. We also prove that if a radial function on a harmonic manifold $M$ satisfies \eqref{1}, \eqref{2} or \eqref{3}, then $M$ must be $\mathbb R^n$, $\mathbb S^n$, $\mathbb H^n$, $\mathbb CH^n$ or $\mathbb QH^n$. Finally, we show that if the mean curvature of a geodesic sphere in a harmonic manifold $M$ is the same as that in a space form, $\mathbb CH^n$ or $\mathbb QH^n$, then $M$ is the space form, $\mathbb CH^n$ or $\mathbb QH^n$, respectively.

\section{Laplacian}

The radial Green's functions of $\mathbb R^n$, $\mathbb S^n$ and $\mathbb H^n$ are $ \frac{1}{(2-n)n\omega_n}r^{2-n}$ ($\frac{1}{2\pi}\log r$ if $n=2,\omega_n=$  volume of a unit ball in $\mathbb R^n$) and $G(r)$ such that $G'(r)=\frac{1}{n\omega_n}\sin^{1-n}r$, $G'(r)=\frac{1}{n\omega_n}\sinh^{1-n}r$, respectively.
 \begin{thm}\label{thm:1}
Let $G_p(r)$ be a nonconstant radial harmonic function on a punctured neighborhood of $p$ in a simply connected harmonic manifold $M^n$ with $r(\cdot)={\rm dist}(p,\cdot)$. If $G_p(r)$ is the same as the radial Green's function of a space form, $\mathbb CH^n$ or $\mathbb QH^n$ at every point $p\in M$, then $M$ is the space form, $\mathbb CH^n$ or $\mathbb QH^n$, respectively.
 \end{thm}
\begin{proof}
Let $\delta_p$ be the Dirac delta function centered at $p\in M$. Integrate $\Delta G_p(r)=\delta_p$ over a geodesic ball $D_r$ of radius $r$ with center at $p$:
$$1=\int_{D_r}\Delta G_p(r)=\int_{\partial D_r}G_p'(r).$$
Hence $${\rm vol}(\partial D_r)=\frac{1}{G_p'(r)}\,\,\,\,{\rm and}\,\,\,\,{\rm vol}(D_r)=\int_0^r\frac{1}{G_p'(r)}=\int_{{\rm exp}_p^{-1}(D_r)}\omega_p(r).$$
Then $M$ should have the same volume density $\omega_p(r)$ as a space form, $\mathbb CH^n$, or $\mathbb QH^n$. Therefore by Shah \cite{Sh}, Szab\'o \cite{Sz}, Ramachandran-Ranjan \cite{RR}{{,}} $M$ is $\mathbb R^n, \mathbb S^n, \mathbb H^n$, $\mathbb CH^n$ or $\mathbb QH^n$, respectively.
\end{proof}
\begin{cor}
If $\Delta r^2=2n$  for $r(\cdot)={\rm dist}(p,\cdot)$ at any point $p$ of a harmonic manifold $M^n$, then $M$ is flat.
\end{cor}
\begin{proof}
It is known that $$\Delta f^k=k(k-1)f^{k-2}|\nabla f|^2+kf^{k-1}\Delta f.$$
Setting $f=r^2$ and $k=1-n/2$, $n\neq2$, one can compute that
$$\Delta r^{2-n}=0.$$
Hence $M$ has a radial harmonic function $ \frac{1}{(2-n)n\omega_n}r^{2-n}$ which is the same as Green's function of $\mathbb R^n$. Therefore the conclusion follows from Theorem \ref{thm:1}. The proof for $n=2$ is similar.
\end{proof}
The condition (3) in Introduction says that the mean curvature of a small geodesic sphere in a harmonic manifold is constant. The following theorem characterizes a harmonic manifold in terms of the mean curvature.
  \begin{thm}
Let $H(r)$ be the mean curvature of a geodesic sphere of radius $r$ in a simply connected harmonic manifold $M$. If $H(r)$ is the same as that in a space form, $\mathbb CH^n$ or $\mathbb QH^n$ for any point $p\in M$ with $r(\cdot)={\rm dist}(p,\cdot)$, then $M$ is the space form, $\mathbb CH^n$ or $\mathbb QH^n$, respectively.
 \end{thm}
\begin{proof}
Let $\gamma$ be a geodesic from $p$ parametrized by arclength $r$ with $\gamma(0)=p$ in a Riemannian manifold $M^n$. Let $\{e_1,\ldots,e_n\}$ be an orthonormal frame at $\gamma(0)$ with $e_1=\gamma'(0)$ and extend it to a parallel orthnormal frame field $\{e_1(r),\ldots,e_n(r)\}$ along $\gamma(r)$ with $e_i(0)=e_i$. Define $Y_i(r),i=2,\ldots,n,$ to be the Jacobi field along $\gamma(r)$ satisfying $Y_i(0)=0$ and $Y_i'(0)=e_i$.
If $M$ is harmonic, then
\begin{equation}\label{Y}
\omega_p(r)=\frac{1}{r^{n-1}}\sqrt{{\rm det}\langle Y_i(r),Y_j(r)\rangle}:=\frac{1}{r^{n-1}}\Theta(r).
\end{equation}
In other words, the volume form $dV$ of $M$ in normal coordinates $x_1,\ldots,x_n$ becomes
$$dV=\omega_p(r)dx_1\cdots dx_n=\Theta(r)\, dr\,dA,$$
where $dA$ is the volume form on the unit sphere in $\mathbb R^n$. Since the volume of a geodesic sphere $\partial D_r$ is $\int_{S}\Theta(r)$ ($S$: unit sphere in $\mathbb R^n$), the first variation of area on the geodesic sphere $\partial D_r$ yields
\begin{equation}\label{H}
H(r)=\frac{\Theta'(r)}{\Theta(r)}.
\end{equation}
As $H(r)$ is the same as that of a space form, $\Theta(r)$ must be the same as that of the space form, and so $\omega_p(r)$ is the same as the volume density function of the space form. Similarly for $\mathbb CH^n$ and $\mathbb QH^n$ with $n$ replaced by $2n$ and $4n$, respectively. Therefore Shah, Szab\'o and Ramachandran-Ranjan's theorems complete the proof.
\end{proof}

\section{Eigenfunctions}

In \eqref{Y} $Y_i(r)$ has a Taylor series expression
$$Y_i(r)=e_i(r)r-\frac{1}{6}R(e_i(r),e_1(r))e_1(r)r^3+o(r^3).$$
Hence
$$\langle Y_i(r),Y_j(r)\rangle=r^2(\delta_{ij}-\frac{1}{3}\langle R(e_i(r),e_1(r))e_1(r),e_j(r)\rangle r^2+o(r^2))$$
and
$${\rm det}\langle Y_i(r),Y_j(r)\rangle=r^{2n-2}{\rm det}\left(I_{n-1}-\frac{1}{3}R_{i11j}(\gamma(r))r^2+o(r^2)\right).$$
If $M$ is harmonic, then
\begin{equation}\label{Ledger}
\frac{d^2}{dr^2}|_{r=0}\,\omega_p(r)=\frac{d^2}{dr^2}|_{r=0}\,\left(\frac{1}{r^{n-1}}\sqrt{{\rm det}\langle Y_i(r),Y_j(r)\rangle}\right)=-\frac{1}{3}Ric(p),
\end{equation}
which is called {\it Ledger's formula} (\cite{B}, p.161).
This formula implies that harmonic manifolds are Einstein.

\begin{thm}\label{thm:3}
{\rm a)} If $\Delta\cos r=-n\cos r$ on a complete simply connected harmonic manifold $M^n$ at any point $p\in M$ with $r(\cdot)={\rm dist}(p,\cdot)$, then $M=\mathbb S^n$.\\
{\rm b)} If $\Delta\cosh r=n\cosh r$ on a complete simply connected harmonic manifold $M^n$ at any point $p\in M$ with $r(\cdot)={\rm dist}(p,\cdot)$, then $M=\mathbb H^n$.
\end{thm}
\begin{proof}
a) Since $\Delta\cos r=-n\cos r$, it is not difficult to show
\begin{equation}\label{r}
\Delta r=(n-1)\cot r.
\end{equation}
Let $G_p(r)$ be the radial function on $M$ such that $G_p'(r)=\frac{1}{n\omega_n}\sin^{1-n}r$.
Then
\begin{eqnarray*}
\Delta G_p(r)&=&{\rm div}\nabla G_p(r)={\rm div}(\frac{1}{n\omega_n}\sin^{1-n}r\nabla r)\\
&=&\frac{(1-n)}{n\omega_n}\sin^{-n}r\cos r|\nabla r|^2+\frac{1}{n\omega_n}\sin^{1-n}r\Delta r\\&=&0.\,\,\,\,\,\,\,({\rm by} \,\,\eqref{r})
\end{eqnarray*}
Theorem \ref{thm:1} completes the proof.

({\it Another proof~}) It is easy to show that for a radial function $f$ on a harmonic manifold $M$
\begin{equation}\label{laplacian}
\Delta f=\frac{d^2f}{dr^2}+H(r)\frac{df}{dr},
\end{equation}
where $H(r)$ is the mean curvature of $\partial D_r$. Hence from \eqref{H} and \eqref{r} one gets for $f(r):=r$
$$\frac{\Theta'(r)}{\Theta(r)}=H=(n-1)\cot r.$$
Therefore $$\Theta(r)=\sin^{n-1}r\,\,\,\,{\rm and}\,\,\,\,\omega_p(r)=\frac{1}{r^{n-1}}\sin^{n-1}r.$$
Then
\begin{eqnarray*}
\omega_p'(r)&=&(n-1)\left(\frac{\sin r}{r}\right)^{n-2}\left(\frac{\sin r}{r}\right)',\\
\omega_p''(r) &=& (n-1)(n-2)\left(\frac{\sin r}{r}\right)^{n-3}\left(\left(\frac{\sin
r}{r}\right)'\right)^2\\
&&{{+(n-1)\left(\frac{\sin r}{r}\right)^{n-2}\left(\frac{\sin
r}{r}\right)''}}.
\end{eqnarray*}
Hence  Ledger's formula \eqref{Ledger} implies
$$Ric(p)=-3\,\frac{d^2}{dr^2}|_{r=0}~\omega_p(r)=n-1$$
for any $p\in M$.
Using the Riccati equation for the second fundamental form $h$ on the geodesic sphere, one obtains
\begin{eqnarray*}
Ric(M)&=&-{\rm tr}h'-{\rm tr}h^2\\
&\leq&(n-1)\csc^2r-(n-1)\cot^2r\,\,\,(\because {\rm tr}h^2\geq\frac{1}{n-1}({\rm tr}h)^2) \\
&=&n-1.
\end{eqnarray*}
Since equality holds above, one should have ${\rm tr}h^2=\frac{1}{n-1}({\rm tr}h)^2$. Hence the linear operator $h$ is a multiple of the identity, meaning that every geodesic sphere is umbilic. So the sectional curvature is constant on the geodesic sphere. Therefore $M=\mathbb S^n$ as $M$ is Einstein.

Proof of b) is similar to a).
\end{proof}

\begin{thm}
{\rm a)} Let $f(r):=1+\frac{n+1}{n}\sinh^2r$ be a radial function on a complete \, simply connected \, K\"ahler hamonic \, manifold $M^{2n}$. If {{$\Delta f=4(n+1)f$}} at any point $p\in M$ with $r(\cdot)={\rm dist}(p,\cdot)$, then $M$ is isometric to the complex hyperbolic space $\mathbb CH^n$.\\
{\rm b)} Let $f(r):=1+\frac{n+1}{n}\sinh^2r$ be a radial function on a complete simply connected quaternionic K\"ahler hamonic manifold $M^{4n}$. If $\Delta f=8(n+1)f$ at any point $p\in M$ with $r(\cdot)={\rm dist}(p,\cdot)$, then $M$ is isometric to the quaternionic hyperbolic space $\mathbb QH^n$.
\end{thm}

\begin{proof}
a) \eqref{laplacian} and \eqref{H} yield
$$\Delta f=f''+\frac{\Theta'}{\Theta}f'=4(n+1)f.$$
Hence for $f(r)=1+\frac{n+1}{n}\sinh^2r$ one can compute
$$\frac{\Theta'(r)}{\Theta(r)}=(2n-1)\coth r +\tanh r.$$
Therefore
$$\Theta(r)=\sinh^{2n-1}\cosh r\,\,\,\, {\rm and} \,\,\,\,\omega_p(r)=\frac{1}{r^{2n-1}}\sinh^{2n-1}\cosh r.$$
Thus the theorem follows from Ramachandran-Ranjan's theorem \cite{RR}.

b) For $f(r)=1+\frac{n+1}{n}\sinh^2r$
$$\frac{\Theta'(r)}{\Theta(r)}=(4n-1)\coth r +3\tanh r \,\,\,\,{\rm and} \,\,\,\, \Theta(r)=\sinh^{4n-1}r\cosh^3 r .$$
Hence $\omega_p(r)=\frac{1}{r^{4n-1}}\sinh^{4n-1}r\cosh^3r$, which is the same as the volume density of $\mathbb QH^n$.
\end{proof}


\end{document}